\documentclass[11pt]{article}

\usepackage{amsmath,epsfig,amssymb,amsbsy,verbatim,array,color,graphics,graphicx}
\usepackage{amssymb,amsmath,amsthm}
\usepackage{graphicx}
\usepackage{subdepth}

\newtheorem{thm}{Theorem}
\newtheorem{lem}[thm]{Lemma}

\newtheorem{obs}[thm]{Observation}
\newtheorem{conj}[thm]{Conjecture}
\newtheorem{clm}[thm]{Claim}

\def\dcup{\,\dot\cup\,}

\begin{document}

\title{Monochromatic $k$-connection of graphs\\[2ex]}

\author{Qingqiong Cai$^\dag$
\and
Shinya Fujita$^\ddag$
\and Henry Liu%
\thanks{Corresponding author\newline\indent\hspace{0.12cm}$^\dag$College of Computer Science, Nankai University, Tianjin 300350, China. E-mail address: caiqingqiong@nankai.edu.cn\newline\indent\hspace{0.12cm}$^\ddag$School of Data Science, Yokohama City University, Yokohama 236-0027, Japan. E-mail address: fujita@yokohama-cu.ac.jp\newline\indent\hspace{0.12cm}$^\S$School of Mathematics, Sun Yat-sen University, Guangzhou 510275, China. E-mail address: liaozhx5@mail.sysu.edu.cn\newline\indent\hspace{0.12cm}$^\P$Department\, of\, Mathematics,\, Ajou\, University,\, Suwon\, 16499,\, Republic\, of\, Korea.\, E-mail\, address: borampark@ajou.ac.kr}\:\,$^\S$
\and Boram Park$^\P$\\[2ex]
}

\date{14 February 2024}
\maketitle
\begin{abstract}
An edge-coloured path is \emph{monochromatic} if all of its edges have the same colour. For a $k$-connected graph $G$, the \emph{monochromatic $k$-connection number} of $G$, denoted by $mc_k(G)$, is the maximum number of colours in an edge-colouring of $G$ such that, any two vertices are connected by $k$ internally vertex-disjoint monochromatic paths. In this paper, we shall study the parameter $mc_k(G)$. We obtain bounds for $mc_k(G)$, for general graphs $G$. We also compute $mc_k(G)$ exactly when $k$ is small, and $G$ is a graph on $n$ vertices, with a spanning $k$-connected subgraph having the minimum possible number of edges, namely $\lceil\frac{kn}{2}\rceil$. We prove a similar result when $G$ is a bipartite graph. \\


\end{abstract}

\section{Introduction}\label{introsect}

In this paper, all graphs considered are finite, simple and undirected. For a graph $G$ and a vertex $x\in V(G)$, the \emph{neighbourhood} of $x$ is $N_G(x)=\{y\in V(G):xy\in E(G)\}$, and the \emph{degree} of $x$ is $\deg_G(x)=|N_G(x)|$. Let $\delta(G)$ and $\Delta(G)$ denote the \emph{minimum degree} and \emph{maximum degree} of $G$. For non-empty, disjoint subsets $X,Y\subset V(G)$, let $E_G(X,Y)$ denote the set of edges of $G$ with one end-vertex in $X$, and the other in $Y$. A \emph{super-path} is a path with length at least two. Throughout, let $k$ and $r$ be positive integers. A graph $G$ is \emph{$k$-connected} if $|V(G)|\ge k+1$, and $G-U$ is connected for every $U\subset V(G)$ with $|U|\le k-1$. An \emph{edge-colouring} of $G$, or simply a \emph{colouring}, is a function $\phi :E(G)\to [r]$ for some $r$, where $[r]=\{1,\dots,r\}$. We think of $[r]$ as a set of colours, and every edge is given one of the $r$ possible colours. When $G$ is given an edge-colouring, the \emph{colour neighbourhood} of $x\in V(G)$, denoted by $N_G^c(x)$, is the set of colours among all edges incident to $x$, and the \emph{colour degree} of $x$ is $\deg_G^c(x)=|N_G^c(x)|$. The \emph{minimum colour degree} of $G$ is denoted by $\delta^c(G)$. The \emph{subgraph of $G$ induced by a colour $i$} is the graph $G_i$ with vertex set $V(G_i)=V(G)$ and containing all edges with colour $i$. A \emph{monochromatic component} of $G$ is a connected component of $G_i$ for some colour $i$. For any other undefined terms in graph theory, we refer to the book \cite{Bol98}.

An edge-coloured path is \emph{monochromatic} if all of its edges have the same colour. For a connected graph $G$, the \emph{monochromatic connection number} of $G$, denoted by $mc(G)$, is the maximum number of colours in an edge-colouring of $G$ so that, any two vertices are connected by a monochromatic path. This notion of graph connectivity was introduced by Caro and Yuster \cite{CY11} in 2011. Since then, the topic of monochromatic connection of graphs has attracted considerable interest. We refer the reader to the survey paper of Li and Wu \cite{LiWu18} for further information.

The parameter $mc(G)$ is the natural opposite of the \emph{rainbow connection number} $rc(G)$ of $G$, introduced by Chartrand et al.~\cite{CJMZ08} in 2008. For a connected graph $G$, the parameter $rc(G)$ is the minimum number of colours in an edge-colouring of $G$ so that, any two vertices are connected by a \emph{rainbow} path (i.e., a path with distinct colours). In 2009, Chartrand et al.~\cite{CJMZ09} introduced an extension of the parameter $rc(G)$ which considers graphs with higher vertex-connectivity, as follows. A set of internally vertex-disjoint paths connecting two vertices in a graph will simply be called  \emph{disjoint}. By Menger's theorem \cite{Men27}, a graph is $k$-connected if and only if any two vertices are connected by $k$ disjoint paths. For a $k$-connected graph $G$, the \emph{rainbow $k$-connection number} $rc_k(G)$ of $G$ is the minimum number of colours in an edge-colouring of $G$ such that, any two vertices are connected by $k$ disjoint rainbow paths.

Here, we shall consider the natural opposite of the parameter $rc_k(G)$, which will also be an extension of the parameter $mc(G)$. For a $k$-connected graph $G$, an edge-colouring of $G$ is \emph{monochromatic $k$-connected} if any two vertices are connected by $k$ disjoint monochromatic paths. The \emph{monochromatic $k$-connection number} of $G$, denoted by $mc_k(G)$, is the maximum number of colours in a monochromatic $k$-connected colouring of $G$. Clearly $mc_1(G)=mc(G)$, and $mc_k(G)$ is well-defined if and only if $G$ is $k$-connected.

We note that very recently, Li and Li \cite{LL20} considered a closely related parameter. A graph $G$ is \emph{$k$-edge-connected} if $|V(G)|\ge 2$, and $G-E$ is connected for every $E\subset E(G)$ with $|E|\le k-1$. By the $k$-edge-connected version of Menger's theorem \cite{Men27}, a graph is $k$-edge-connected if and only if any two vertices are connected by $k$ edge-disjoint paths. Let $G$ be a $k$-edge-connected graph. An edge-colouring of $G$ is \emph{monochromatic $k$-edge-connected} if any two vertices of $G$ are connected by $k$ edge-disjoint monochromatic paths. The \emph{monochromatic $k$-edge-connection number} of $G$, denoted by $emc_k(G)$, is the maximum number of colours in a monochromatic $k$-edge-connected colouring of $G$. Thus, $emc_k(G)$ is well-defined if and only if $G$ is $k$-edge-connected. Clearly if $G$ is a $k$-connected graph, we have $mc_k(G)\le emc_k(G)$.

This paper will be organised as follows. In Section \ref{gensect}, we prove some bounds for $mc_k(G)$, for any $k$-connected graph $G$ where $k\ge 2$. This result implies that $mc_k(G)=1$ when $e(G)$ is minimum, i.e., $e(G)= \lceil\frac{kn}{2}\rceil$. In Section \ref{smallksect}, we compute $mc_k(G)$ exactly when $k$ is small, and $G$ contains a spanning $k$-connected subgraph with $\lceil\frac{kn}{2}\rceil$ edges. We also prove a similar result when $G$ is a bipartite graph. In Section \ref{compsect}, we deduce some results about $mc_k(G)$ when $G$ is a complete graph, and a complete bipartite graph. Several conjectures will be mentioned along the way.

\section{Monochromatic $k$-connection number of general graphs}\label{gensect}

In this section, we shall investigate the monochromatic $k$-connection number of general graphs. Caro and Yuster \cite{CY11} proved several results about the parameter $mc(G)$. One of their main results is the following.

\begin{thm}\label{CYthm}\textup{\cite{CY11}}
Let $G$ be a connected graph on $n\ge 2$ vertices with chromatic number $\chi(G)$. Then
\begin{equation}\label{CYthmeq}
e(G)-n+2\le mc(G) \le e(G)-n+\chi(G).
\end{equation}
\end{thm}
The lower bound of (\ref{CYthmeq}) can be seen by considering an edge-colouring of $G$ where a spanning tree (with $n-1$ edges) is given one colour, and all remaining edges are given further distinct colours. This lower bound is attained when $G$ is a tree. Caro and Yuster also proved some other sufficient conditions on $G$ so that $mc(G)=e(G)-n+2$ (see \cite{CY11}, Theorem 1). For the upper bound of (\ref{CYthmeq}), we see that it is attained when $G$ is the complete graph $K_n$. Very recently, Jin et al.~\cite{JLY20} characterised all connected graphs $G$ on $n$ vertices such that $mc(G) = e(G)-n+\chi(G)$.

When $G$ is a $k$-connected graph, where $k\ge 2$, the situation for $mc_k(G)$ appears to be rather different, when compared to (\ref{CYthmeq}). Consider the analogous situation to the lower bound of (\ref{CYthmeq}). A subgraph $H\subset G$ is a \emph{minimum spanning $k$-connected subgraph} if $H$ is a $k$-connected subgraph such that $V(H)=V(G)$, and $e(H)$ is minimum. We define
\[
h_k(G)=\max\{mc_k(H):H\textup{ is a minimum spanning $k$-connected subgraph of }G\}.
\]

\begin{obs}\label{LBobs}
Let $G$ be a $k$-connected graph, where $k\ge 2$, and let $H\subset G$ be a minimum spanning $k$-connected  subgraph. Then
\begin{equation}\label{mckLBeq}
mc_k(G)\ge e(G)-e(H)+h_k(G)\ge e(G)-e(H)+1.
\end{equation}
\end{obs}
To obtain the first inequality of (\ref{mckLBeq}), we consider the following monochromatic $k$-connected colouring of $G$. We choose $H\subset G$ so that $H$ can be given a monochromatic $k$-connected colouring with $h_k(G)$ colours, and all edges of $E(G)\setminus E(H)$ are given further distinct colours. We propose the following conjecture, which claims that such a colouring of $G$ attains $mc_k(G)$.

\begin{conj}\label{mckconj}
Let $G$ be a $k$-connected graph, where $k\ge 2$, and let $H\subset G$ be a minimum spanning $k$-connected  subgraph. Then
\[
mc_k(G)=e(G)-e(H)+h_k(G).
\]
\end{conj}

We note that Li and Li \cite{LL20} made the following analogous conjecture about the parameter $emc_k(G)$.
\begin{conj}\label{emckconj}\emph{\cite{LL20}}
Let $G$ be a $k$-edge-connected graph, where $k\ge 2$, and let $H\subset G$ be a spanning $k$-edge-connected  subgraph with $e(H)$ minimum. Then
\begin{equation}
emc_k(G)=e(G)-e(H)+\bigg\lfloor\frac{k}{2}\bigg\rfloor.\label{LLconjeq}
\end{equation}
\end{conj}
They verified some cases of Conjecture \ref{emckconj} as follows.
\begin{thm}\label{LLthm}\emph{\cite{LL20}}
Conjecture \ref{emckconj} is true in the following cases.
\begin{enumerate}
\item[(a)] $k=2$.
\item[(b)] $G=K_{k+1}$ where $k\ge 3$.
\item[(c)] $G=K_{k,n}$ where $n\ge k\ge 3$.
\end{enumerate}
\end{thm}

Since Observation \ref{LBobs} has a lower bound for the parameter $mc_k(G)$, we shall focus on the upper bound. We first make the following observation, which allows us to simplify the edge-colourings of $G$ that we will consider.
\begin{obs}\label{UBobs}
Let $G$ be a $k$-connected graph. Suppose we would like to prove that $mc_k(G)\le M$. Then, it suffices to show that for every monochromatic $k$-connected colouring of $G$ where every colour induces exactly one non-trivial component, there are at most $M$ colours.
\end{obs}

Indeed, suppose that $G$ is given a monochromatic $k$-connected colouring, using exactly $r'$ colours. If the subgraph induced by  some colour has at least two non-trivial components, then we may recolour one of the components with a new colour. We may then repeat this recolouring procedure until we obtain a colouring where every colour induces exactly one non-trivial component. We obtain a new colouring of $G$, using exactly $r$ colours for some $r\ge r'$, which is still monochromatic $k$-connected. Thus, if we can show that $r\le M$, then $r'\le r\le M$, and this implies that $mc_k(G)\le M$.


%


Next, we note that for $n>k\ge 2$, any $k$-connected graph $G$ on $n$ vertices has at least $\lceil\frac{kn}{2}\rceil$ edges, since $\delta(G)\ge k$. Harary \cite{Har62} gave examples of $k$-connected graphs which show that the value $\lceil\frac{kn}{2}\rceil$ is best possible for the minimum number of edges.

\begin{thm}\label{Harthm} \textup{\cite{Har62}}
Let $n>k\ge 2$, and $G$ be a $k$-connected graph on $n$ vertices. Then $e(G)\ge\lceil\frac{kn}{2}\rceil$. Moreover, there exists a $k$-connected graph $H_{n,k}$ on $n$ vertices with $e(H_{n,k})=\lceil\frac{kn}{2}\rceil$.
\end{thm}

In the following result, we obtain an upper bound for $mc_k(G)$, as well as the answer for $mc_k(G)$ when $G$ has the minimum possible number of edges.

\begin{thm}\label{gengraphsthm1}
Let $n>k\ge 2$, and $G$ be a $k$-connected graph on $n$ vertices.
\begin{enumerate}
\item[(a)] We have
\begin{equation}
mc_k(G)\le e(G)-\bigg\lceil\frac{k{n\choose 2}-e(G)}{n-2}\bigg\rceil+1.\label{gengraphscoreq}
\end{equation}
\item[(b)] Suppose that $G$ has the minimum possible number of edges, i.e., $e(G)=\lceil\frac{kn}{2}\rceil$. Then $mc_k(G)=1$.
\end{enumerate}
\end{thm}

We note that if $G$ has a minimum spanning $k$-connected subgraph with $\lceil\frac{kn}{2}\rceil$ edges, then by Theorem \ref{gengraphsthm1}(b), we have $h_k(G)=1$. Thus, we have the following version of Conjecture \ref{mckconj}.

\begin{conj}\label{mckconj1}
Let $n>k\ge 2$. Let $G$ be a $k$-connected graph on $n$ vertices, and $H\subset G$ be a minimum spanning $k$-connected  subgraph. If $e(H)=\lceil\frac{kn}{2}\rceil$, then
\[
mc_k(G)=e(G)-\bigg\lceil\frac{kn}{2}\bigg\rceil+1.
\]
\end{conj}

To prove Theorem \ref{gengraphsthm1}, we shall prove another result. Let $G$ be a graph. For a function  $f: {V(G)\choose 2} \rightarrow \mathbb{Z}_{\ge 0}$, we denote $f(\{u,v\})$ by $f(u,v)$ for simplicity. We define the \emph{weight} of $f$ to be
\[
w(f)=\sum_{\{u,v\}} f(u,v),
\]
where throughout, $\sum_{\{u,v\}}$ means the sum is taken over all $\{u,v\}\in {V(G)\choose 2}$. Let
\[
m_G(u,v)=
\left\{
\begin{array}{l@{\quad\:}l}
\min(\deg_G(u),\deg_G(v))-1, & \textup{if }uv\in E(G),\\[0.5ex]
\min(\deg_G(u),\deg_G(v)), & \textup{if }uv\not\in E(G).
\end{array}
\right.
\]
We shall prove the following theorem.

\begin{thm}\label{gengraphsthm2}
Let $G$ be a graph on $n\ge 3$ vertices, with an edge-colouring using exactly $r$ colours. Let $f: {V(G)\choose 2} \rightarrow \mathbb{Z}_{\ge 0}$. For $u,v\in V(G)$, suppose that there are $f(u,v)$ disjoint monochromatic super-paths connecting $u$ and $v$. Then
\begin{equation}
e(G)\ge \bigg\lceil \frac{w(f)}{n-2}\bigg\rceil + r-1.\label{gengraphsthmeq}
\end{equation}
\end{thm}

Assuming Theorem \ref{gengraphsthm2}, we may easily obtain Theorem \ref{gengraphsthm1}.

\begin{proof}[Proof of Theorem \ref{gengraphsthm1}]
(a) Let $G$ be given a monochromatic $k$-connected colouring, using exactly $r$ colours. Then with $f$ as defined in Theorem \ref{gengraphsthm2}, we have
\[
f(u,v) \ge
\left\{
\begin{array}{l@{\quad\:}l}
k, & \textup{if $uv\not\in E(G)$,}\\[0.5ex]
k-1, & \textup{if $uv\in E(G)$,}
\end{array}
\right.
\]
which implies that $w(f)=\sum_{\{u,v\}}f(u,v)\ge k{n\choose 2}-e(G)$. From (\ref{gengraphsthmeq}), we have
\[
r\le e(G)- \bigg\lceil \frac{w(f)}{n-2}\bigg\rceil +1\le e(G)-\bigg\lceil\frac{k{n\choose 2}-e(G)}{n-2}\bigg\rceil+1,
\]
and (\ref{gengraphscoreq}) follows.\\[1ex]
\indent(b) We set $e(G)=\lceil\frac{kn}{2}\rceil$ in (\ref{gengraphscoreq}). If $kn$ is even, we have
\[
mc_k(G)\le\frac{kn}{2}-\bigg\lceil\frac{k{n\choose 2}-\frac{kn}{2}}{n-2}\bigg\rceil+1=1.
\]
Otherwise, we have $k$ and $n$ are both odd, so that $e(G)=\frac{kn+1}{2}$. Thus,
\begin{align*}
mc_k(G) &\le\bigg\lceil\frac{kn}{2}\bigg\rceil-\bigg\lceil\frac{k{n\choose 2}-\frac{kn+1}{2}}{n-2}\bigg\rceil+1=\bigg\lceil\frac{kn}{2}\bigg\rceil-\bigg\lceil\frac{kn+1}{2}-\frac{n-1}{2n-4}\bigg\rceil+1=1,
\end{align*}
since $n>k\ge 3$, so that $n\ge 5$, and $0<\frac{n-1}{2n-4}<1$.
\end{proof}

Now, we will prove Theorem \ref{gengraphsthm2}. We use induction on $r$. The following lemma considers the base case $r=1$, as well as some consequences when equality holds for this case.

\begin{lem}\label{gengraphslem}
Let $G$ be a graph on $n\ge 3$ vertices.
\begin{enumerate}
\item[(a)] Let $f: {V(G)\choose 2} \rightarrow \mathbb{Z}_{\ge 0}$. For $u,v\in V(G)$, suppose that there are $f(u,v)$ disjoint super-paths connecting $u$ and $v$. Then
\begin{equation}\label{gengraphslemeq1}
e(G)\ge \bigg\lceil \frac{w(f)}{n-2}\bigg\rceil.
\end{equation}
\item[(b)] Suppose that $e(G)\ge 1$, and equality holds in (\ref{gengraphslemeq1}). Then $\Delta(G)-\delta(G)\le 1$, and one of the following holds.
\begin{enumerate}
\item[(i)] $G$ is a regular graph, and $\Delta(G)=\delta(G)\ge 2$.
\item[(ii)] $G$ has exactly one vertex of degree $\delta(G)$, and $\Delta(G)\ge 3$, $\delta(G)\ge 2$.
\item[(iii)] $G$ has exactly one vertex of degree $\Delta(G)$, and $\Delta(G)\ge 4$, $\delta(G)\ge 3$.
\end{enumerate}
Moreover, we have
\begin{equation}
e(G)=\bigg \lceil \frac{w(f)}{n-2}\bigg\rceil=\bigg \lceil \frac{w(m_G)}{n-2}\bigg\rceil=\frac{w(m_G)+s}{n-2},\label{gengraphslemeq2}
\end{equation}
where $s=0$ if $G$ is regular, and $s=\frac{n-1}{2}$ otherwise.
\end{enumerate}
\end{lem}

\begin{proof}
(a) For $v\in V(G)$, we have
\[
\deg_G(v) \ge
\left\{
\begin{array}{l@{\quad\:}l}
f(u,v), & \textup{if $uv\not\in E(G)$,}\\[0.5ex]
f(u,v)+1, & \textup{if $uv\in E(G)$.}
\end{array}
\right.
\]
Thus,
\begin{align*}
(n-1)\deg_G(v) &\ge \deg_G(v)+\sum_{u:u\neq v}f(u,v)\\
\deg_G(v) &\ge \frac{\sum_{u:u\neq v}f(u,v)}{n-2}.
\end{align*}
We have
\[
e(G)=\frac{1}{2}\sum_{v\in V(G)}\deg_G(v) \ge \frac{1}{2}\sum_{v\in V(G)} \frac{\sum_{u:u\neq v} f(u,v)}{n-2}= \frac{w(f)}{n-2},
\]
which implies (\ref{gengraphslemeq1}).\\[1ex]
\indent(b) Suppose that $e(G)\ge 1$, and equality holds in (\ref{gengraphslemeq1}). If $\Delta(G)=1$, then $G$ does not contain a super-path, so that $w(f)=0$, and equality in (\ref{gengraphslemeq1}) cannot hold. Hence, $\Delta(G)\ge 2$.

Now, let $V(G)=\{v_1,\dots,v_n\}$, and $d_1\le\cdots\le d_n$ be the degree sequence of $G$, where $d_i=\deg_G(v_i)$ for all $1\le i\le n$. Let $\delta=\delta(G)=d_1$ and $\Delta=\Delta(G)=d_n$. For every $1\le i\le n$, we have
\begin{align*}
\sum_{u:u\neq v_i}m_G(u,v_i) &=\sum_{j\neq i}\min(d_j,d_i)-d_i=  \sum_{j<i}d_j+\sum_{j>i}d_i  -d_i.\\
&=\sum_{j<i}d_j+(n-i-1)d_i.
\end{align*}
Since $f(u,v)\le m_G(u,v)$ for all $\{u,v\}\in{V(G)\choose 2}$,
\begin{align*}
2w(f) &=\sum_{i=1}^n\sum_{u:u\neq v_i}f(u,v_i)\le \sum_{i=1}^n\sum_{u:u\neq v_i}m_G(u,v_i)= \sum_{i=1}^n\bigg(\sum_{j<i}d_j+(n-i-1)d_i\bigg)\\
&=\sum_{i=1}^n((n-i)d_i+(n-i-1)d_i)=\sum_{i=1}^n((n-2)d_i+(n-2i+1)d_i)\\
&=2(n-2)e(G)-\sum_{i=1}^{\lfloor n/2\rfloor}(n-2i+1)(d_{n+1-i}-d_i).
\end{align*}
If $\Delta-\delta=d_n-d_1\ge 2$, then $2w(f)\le 2(n-2)e(G)-2(n-1)$, so $e(G)>\frac{w(f)}{n-2}+1$. Thus, we have $\Delta-\delta\le 1$. Then, if $n\ge 4$ and $d_n-d_1=d_{n-1}-d_2=1$, we have $2w(f)\le 2(n-2)e(G)-(n-1)-(n-3)$, so $e(G)\ge \frac{w(f)}{n-2}+1$. Thus, $d_2=\cdots=d_{n-1}$. If $\Delta=\delta\ge 2$, then $G$ is regular, and (i) holds. Now, let $\Delta-\delta=1$. Since $\sum_{i=1}^nd_i$ is even, if $d_2=\cdots=d_{n-1}=\Delta$, then $\Delta$ must be odd, so that $\Delta\ge 3$, $\delta\ge 2$, and (ii) holds. Otherwise, $d_2=\cdots=d_{n-1}=\delta$, and $\Delta$ must be even. If $\Delta=2$, then $G$ has exactly one super-path, so that $w(f)\le 1$. But $e(G)\ge 2$, so equality in (\ref{gengraphslemeq1}) cannot hold. Hence, $\Delta\ge 4$, $\delta\ge 3$, and (iii) holds.

Finally, note that $f(u,v)\le m_G(u,v)$ for all $\{u,v\}\in{V(G)\choose 2}$, so that $w(f)\le w(m_G)$. If $G$ satisfies (i), then $G$ is $\Delta$-regular, and $e(G)=\frac{\Delta n}{2}$. We have
\[
w(m_G)=\sum_{\{u,v\}}m_G(u,v)=\Delta{n\choose 2}-e(G)=(n-2)e(G),
\]
so that $e(G)=\big\lceil\frac{w(f)}{n-2}\big\rceil\le \big\lceil\frac{w(m_G)}{n-2}\big\rceil=\frac{w(m_G)}{n-2}=e(G)$, and (\ref{gengraphslemeq2}) holds with $s=0$.

If $G$ satisfies (ii), then $e(G)=\frac{\Delta n-1}{2}$. We have
\begin{align*}
w(m_G) &=\sum_{\{u,v\}}m_G(u,v)=\sum_{\{u,v\}}\min(\deg_G(u),\deg_G(v))-e(G)\\
&=(\Delta -1)(n-1)+\Delta{n-1\choose 2}-e(G)=(n-2)e(G)-\frac{n-1}{2}.
\end{align*}
If $G$ satisfies (iii), then $e(G)=\frac{(\Delta-1) n+1}{2}$. We have
\[
w(m_G) =\sum_{\{u,v\}}m_G(u,v)=(\Delta-1){n\choose 2}-e(G)=(n-2)e(G)-\frac{n-1}{2}.
\]
In both cases (ii) and (iii), we have $e(G)=\big\lceil\frac{w(f)}{n-2}\big\rceil\le\big\lceil\frac{w(m_G)}{n-2}\big\rceil=\frac{w(m_G)+s}{n-2}=e(G)$ where $s=\frac{n-1}{2}$, and (\ref{gengraphslemeq2}) holds.
\end{proof}

\begin{proof}[Proof of Theorem \ref{gengraphsthm2}]
It suffices to prove the case of the theorem where in the colouring of $G$, every colour induces exactly one non-trivial component. Indeed, suppose that we have a colouring with exactly $r'$ colours. Applying the recolouring procedure as described after Observation \ref{UBobs}, we obtain a new colouring of $G$ using exactly $r\ge r'$ colours, where every colour induces exactly one non-trivial component. We may then apply the theorem to this new colouring with the same function $f$, to obtain $e(G)\ge\big\lceil\frac{w(f)}{n-2}\big\rceil+r-1\ge \big\lceil\frac{w(f)}{n-2}\big\rceil+r'-1$.

We prove this case of Theorem \ref{gengraphsthm2} by induction on $r$. Lemma \ref{gengraphslem} asserts the base case $r=1$. Now let $r\ge 2$. Suppose that a graph $G$  on $n\ge 3$ vertices is given a colouring with exactly $r$ colours, where the set of colours is $[r]$, and every colour induces exactly one non-trivial component. Assume that Theorem \ref{gengraphsthm2} holds when $G$ is given such a colouring with fewer than $r$ colours. Let $f: {V(G)\choose 2} \rightarrow \mathbb{Z}_{\ge 0}$ be such that for all $u,v\in V(G)$, there are $f(u,v)$ disjoint monochromatic super-paths connecting $u$ and $v$. For $1\le i\le r$, recall that $G_i$ is the spanning subgraph of $G$ induced by colour $i$. Let $H_i=G-E(G_i)$. For simplicity, we denote $m_{G_i}$ by $m_i$. We define functions $f_i,g_i:{V(G)\choose 2}\to\mathbb Z_{\ge 0}$ where
\begin{align*}
f_i(u,v) = &\textup{ Maximum number of disjoint super-paths connecting }u\text{ and }v\text{ in }G_i,\\
g_i(u,v) =&\textup{ Maximum number of disjoint monochromatic super-paths connecting}\\
& \textup{ $u$ and }v\text{ in }H_i.
\end{align*}
For any $1\le i\le r$, by Lemma \ref{gengraphslem}(a), we have
\[
e(G_i)\ge \bigg\lceil \frac{w(f_i)}{n-2}\bigg\rceil.
\]
Since $H_i$ uses exactly $r-1$ colours, by the induction hypothesis, we have
\[
e(H_i)\ge \bigg\lceil \frac{w(g_i)}{n-2}\bigg\rceil +r-2.
\]
Since $f(u,v)\le f_i(u,v)+g_i(u,v)$ for every $\{u,v\}\in {V(G)\choose 2}$, we have
\begin{equation}\label{eq:G}
e(G)=e(G_i)+e(H_i)\ge \left\lceil \frac{w(f_i)}{n-2}\right\rceil + \bigg\lceil \frac{w(g_i)}{n-2}\bigg\rceil +r-2 \ge\left\lceil \frac{w(f)}{n-2}\right\rceil + r-2.
\end{equation}

To obtain a contradiction, suppose that $e(G)= \big \lceil \frac{w(f)}{n-2}\big\rceil +r-2$. Then for every $1\le i\le r$, the equalities of \eqref{eq:G} hold, and $e(G_i) = \big\lceil \frac{w(f_i)}{n-2}\big\rceil$. By Lemma \ref{gengraphslem}(b), we have
$G_i$ satisfies one of (i), (ii) or (iii), and
\begin{equation}\label{eq:Gi2}
e(G_i)= \bigg\lceil \frac{w(f_i)}{n-2} \bigg\rceil = \bigg\lceil \frac{w(m_i)}{n-2} \bigg\rceil=\frac{w(m_i)+s_i}{n-2},
\end{equation}
where $s_i=0$ if $G_i$ is regular, and $s_i=\frac{n-1}{2}$ otherwise.

We first prove three claims. Let $B_1,B_2,B_3$ denote, respectively, the set of colours $i$ such that $G_i$ satisfies (i), (ii) and (iii). Let $b_2=|B_2|$, $b_3=|B_3|$, and $b=b_2+b_3$. Then $s_1+\cdots+s_r = \frac{b(n-1)}{2}$. Let $w(f)+t=w(m_1)+\cdots +w(m_r)$.

\begin{clm}\label{claim:t}
$\frac{b(n-1)}{2}+t<(r-1)(n-2)$.
\end{clm}
\begin{proof}
We have
\begin{align*}
\sum_{i=1}^{r}\frac{w(m_i)+s_i}{n-2} -(r-2) &= e(G) -(r-2)=\bigg\lceil \frac{w(f)}{n-2}\bigg\rceil=\bigg\lceil \frac{\sum_{i=1}^{r}w(m_i)-t}{n-2}\bigg\rceil\\
&<\frac{\sum_{i=1}^{r}w(m_i)-t}{n-2}+1,
\end{align*}
where the first equality follows from (\ref{eq:Gi2}) and $e(G)=\sum_{i=1}^re(G_i)$. Thus $\frac{b(n-1)}{2(n-2)} -(r-2)<-\frac{t}{n-2}+1$, which rearranges to $\frac{b(n-1)}{2}+t<(r-1)(n-2)$.
\end{proof}

Now, note that for all $\{u,v\}\in{V(G)\choose 2}$, we have $f_i(u,v)\le m_i(u,v)$ for all $1\le i\le r$. Thus,
\[
f(u,v)\le \sum_{i=1}^r f_i(u,v)\le \sum_{i=1}^r m_i(u,v).
\]
We define a graph $F$ so that $V(F)=V(G)$, and $uv\in E(F)$ if and only if $f(u,v)< \sum_{i=1}^r m_i(u,v)$.
Then, we have
\begin{align*}
w(f)+t &=\sum_{i=1}^r w(m_i)=\sum_{uv\in E(F)}\sum_{i=1}^rm_i(u,v)+\sum_{uv\not\in E(F)}\sum_{i=1}^rm_i(u,v)\\
&\ge \bigg(e(F)+ \sum_{uv\in E(F)}f(u,v)\bigg)+\sum_{uv\not\in E(F)}f(u,v)=e(F)+w(f),
\end{align*}
so that $e(F)\le t$.

\begin{clm}\label{claim:F}
Let $v\in V(G)$ and $x,y\in N_G(v)$. Let the colours of $vx$ and $vy$ be $i$ and $j$ respectively, where $i\neq j$. If $\deg_{G_i}(x)\le \deg_{G_i}(y)$ and $\deg_{G_j}(x)\ge \deg_{G_j}(y)$, then $xy\in E(F)$.
\end{clm}

\begin{proof}
Suppose that $xy\not\in E(F)$. Then $f(x,y)=\sum_{\ell=1}^rm_\ell(x,y)$. This means that there are $m_i(x,y)$ (resp.~$m_j(x,y)$) super-paths of colour $i$ (resp.~colour $j$) connecting $x$ and $y$, with all $m_i(x,y)+m_j(x,y)$ super-paths disjoint. Since $\deg_{G_i}(x)\le \deg_{G_i}(y)$ and $\deg_{G_j}(x)\ge \deg_{G_j}(y)$, this means that all edges incident to $x$  of colour $i$ and all edges incident to $y$ of colour $j$, except for $xy$ if $xy\in E(G)$, are used in the $m_i(x,y)+m_j(x,y)$ super-paths. Hence, $vx$ and $vy$ are used in the $m_i(x,y)+m_j(x,y)$ super-paths, which is a contradiction since the super-paths are disjoint.
\end{proof}

\begin{clm}\label{clm:Delta}
$\Delta(G_i)\ge 3$ for all $1\le i\le r$.
\end{clm}
\begin{proof}
Clearly from Lemma \ref{gengraphslem}(b), it suffices to show that for all $j\in B_1$, we do not have $\Delta(G_j)=2$. Assume the contrary for some $j\in B_1$. Since $G_j$ has exactly one non-trivial component, we have $G_j$ is a cycle of length $n$. For any two vertices $u,v\in V(G_j)$, there are two disjoint paths in $G_j$ connecting $u$ and $v$, and these two paths together always use all vertices of $G_j$. Thus, $uv\in E(F)$. On the other hand, since $\delta(G_i)\ge 2$ for all $1\le i\le r$, we have $e(G_i)\ge n$. Thus $rn\le {n\choose 2}$, and so $r\le \frac{n-1}{2}$. By Claim \ref{claim:t}, we have ${n\choose 2}=e(F)\le t< (r-1)(n-2) \le \frac{(n-2)(n-3)}{2}$, a contradiction.
\end{proof}

We now complete the proof of the theorem. By Lemma \ref{gengraphslem}(b), for $i\in B_2$ (resp.~$i\in B_3$), let $z_i$ be the unique vertex with minimum (resp.~maximum) degree in $G_i$. We define the set $V'$ where
\[
V'=
\left\{
\begin{array}{l@{\quad\:}l}
V(G)\setminus\{z\}, & \textup{if all $z_i$ ($i\in B_2\cup B_3$) coincide at the same vertex $z\in V(G)$,}\\[0.5ex]
V(G), & \textup{otherwise.}
\end{array}
\right.
\]
We estimate $2e(F)$ as follows. For $x\in V'$, we may choose a colour $i$ such that either $i\in B_1$, or $i\in B_2\cup B_3$ and $x\neq z_i$. Let $v\in N_{G_i}(x)$, and
\[
Y_x=
\left\{
\begin{array}{l@{\quad\:}l}
N_G(v)\setminus(N_{G_i}(v)\cup \{z_i\}), & \textup{if $i\in B_2$,}\\[0.5ex]
N_G(v)\setminus N_{G_i}(v), & \textup{if $i\in B_1\cup B_3$.}
\end{array}
\right.
\]
By Lemma \ref{gengraphslem}(b) and Claim \ref{clm:Delta}, we have $\delta(G_\ell)\ge 3$ for $\ell\in B_1\cup B_3$, and $\delta(G_\ell)\ge 2$ for $\ell\in B_2$. Thus,
\begin{equation}\label{YxLBeq}
|Y_x| \ge \left\{
\begin{array}{l@{\quad\:}l}
3(r-b_2)+2(b_2-1)-1=3r-3-b_2,& \textup{if $i\in B_2$,}\\[0.5ex]
3(r-1-b_2)+2b_2=3r-3-b_2, & \textup{if $i\in B_1\cup B_3$.}
\end{array}
\right.
\end{equation}
Now, let $\Delta=\Delta(G)$, $\Delta_\ell=\Delta(G_\ell)$, and $\delta_\ell=\delta(G_\ell)$, for $1\le \ell\le r$. Let $y\in Y_x$. We have $\deg_{G_i}(x)\le \deg_{G_i}(y)$. Also, $y\in N_{G_j}(v)$ for some $j\neq i$. We have $\deg_{G_j}(x)\ge \deg_{G_j}(y)$, unless if $j\in B_2^x$, where $B_2^x=\{\ell\in B_2: x=z_\ell\}$; or if $j\in B_3$ and $y=z_j$. There are at most $\sum_{\ell\in B_2^x}\Delta_\ell+b_3$ of these exceptional vertices for $y$ in $Y_x$. By Claim \ref{claim:F} and (\ref{YxLBeq}), we have
\[
\deg_F(x)\ge |Y_x|-\sum_{\ell\in B_2^x}\Delta_\ell-b_3\ge 3r-3-b-\sum_{\ell\in B_2^x}\Delta_\ell.
\]
Note that the sets $B_2^x$ are disjoint subsets of $B_2$ as $x$ varies over $V'$, and so $\sum_{x\in V'}\sum_{\ell\in B_2^x} \Delta_\ell \le \sum_{\ell=1}^r\Delta_\ell=\sum_{\ell=1}^r\delta_\ell+b\le\Delta+b$. Thus
\[
2e(F) \ge\sum_{x\in V'}\deg_F(x)\ge\sum_{x\in V'}\bigg(3r-3-b-\sum_{\ell\in B_2^x}\Delta_\ell\bigg)\ge (3r-3-b)(n-1)-\Delta-b,
\]
since $3r-3-b\ge 2r-3\ge 1$, so the final inequality holds.

But now, using $e(F)\le t$ and Claim \ref{claim:t}, we have
\[
2(r-1)(n-2)-b(n-1)>2t\ge 2e(F)\ge(3r-3-b)(n-1)-\Delta-b,
\]
which implies that $\Delta>(r-1)n+r-1-b\ge n-1$, a contradiction. This completes the induction step of the proof.

The proof of the special case of Theorem \ref{gengraphsthm2} is complete. Theorem \ref{gengraphsthm2} itself follows.
\end{proof}

\section{Monochromatic $k$-connection number for small $k$}\label{smallksect}

Our next aim is to consider the monochromatic $k$-connection number for a class of graphs when $k\ge 2$ is small. Firstly, we consider $k$-connected graphs on $n$ vertices where a minimum spanning $k$-connected subgraph has the fewest possible number of edges, which is $\lceil\frac{kn}{2}\rceil$ by Theorem \ref{Harthm}. In the following result, we compute $mc_k(G)$ exactly for such graphs $G$, where $k\in\{2,3,4,5\}$. In other words, we have a solution to Conjecture \ref{mckconj1} for these values of $k$.

\begin{thm}\label{smallkthm}
Let $n>k$, where $k\in\{2,3,4,5\}$. Let $G$ be a $k$-connected graph on $n$ vertices, and $H\subset G$ be a minimum spanning $k$-connected subgraph. If $e(H)=\lceil\frac{kn}{2}\rceil$, then
\[
mc_k(G)=e(G)-\bigg\lceil\frac{kn}{2}\bigg\rceil+1.
\]
\end{thm}

We note that the case $k=2$ follows from Theorem \ref{LLthm}(a). Indeed, if $G$ is $2$-connected, and $H\subset G$ is a minimum spanning $2$-connected subgraph with $e(H)=n$, then $H$ is a Hamilton cycle. Setting $k=2$ in (\ref{LLconjeq}), we have $mc_2(G)\le emc_2(G)=e(G)-n+1$. The matching lower bound follows from Observation \ref{LBobs}. Our proof of Theorem \ref{smallkthm} will include the case $k=2$.

We also consider a version of Theorem \ref{smallkthm} for bipartite graphs. Note that if $G$ is a connected bipartite graph on $n\ge 2$ vertices, then by Theorem \ref{CYthm}, we have $mc(G)=e(G)-n+2$. Thus, we assume that $k\ge 2$ throughout. Let $t\ge s\ge k$, and let $G$ be a $k$-connected bipartite graph with classes $X$ and $Y$, where $|X|=s$ and $|Y|=t$. Note that the condition $t\ge s\ge k$ is necessary for $G$ to be $k$-connected. Now, if $H\subset G$ is a minimum spanning $k$-connected subgraph, then since every vertex of $H$ in $Y$ has degree at least $k$, we have $e(H)\ge kt$. Similar to Theorem \ref{Harthm}, we can show that for every $t\ge s\ge k$, there exists a $k$-connected bipartite graph $H_{s,t,k}$ with class sizes $s$ and $t$, and $e(H_{s,t,k})=kt$. We have the following lemma of Plummer and Saito \cite{PS05}.

\begin{lem}\label{PSlem}\textup{\cite{PS05}}
Let $s\ge k\ge 2$. Then, there exists a $k$-regular, $k$-connected bipartite graph, where both partition classes have $s$ vertices.
\end{lem}

The construction of the bipartite graph in Lemma \ref{PSlem}, given in \cite{PS05}, is as follows. Let $A=\{a_0,\dots,a_{s-1}\}$ and $B=\{b_0,\dots,b_{s-1}\}$ be the partition classes. For every $0\le i\le s-1$, add the edges $a_ib_i,a_ib_{i+1},\dots,a_ib_{i+k-1}$, where the indices are taken modulo $s$.

Hence, we may obtain such a bipartite graph $H_{s,t,k}$ with partition classes $X$ and $Y$ as follows. Let $Y'\subset Y$ be such that $|Y'|=s$. By Lemma \ref{PSlem}, there exists a $k$-regular, $k$-connected bipartite graph $H'$ with classes $X$ and $Y'$. We obtain $H_{s,t,k}$ by taking $H'$, and then for each $v\in Y\setminus Y'$, connect $v$ to $k$ vertices of $X$ (two different vertices of $Y\setminus Y'$ may be connected to two different $k$-subsets of $X$). Then $H_{s,t,k}$ is $k$-connected, and $e(H_{s,t,k})=kt$.

Thus, we propose the version of Conjecture \ref{mckconj1} for bipartite graphs, as follows.

\begin{conj}\label{mckbipconj1}
Let $t\ge s\ge k\ge 2$. Let $G$ be a $k$-connected bipartite graph whose partition classes have $s$ and $t$ vertices, and let $H\subset G$ be a minimum spanning $k$-connected subgraph. If $e(H)=kt$, then
\[
mc_k(G)=e(G)-kt+h_k(G).
\]
\end{conj}

We shall prove the case of Conjecture \ref{mckbipconj1} when $k\in\{2,3\}$. The following partial result can be considered as an analogous result of Theorem \ref{smallkthm} for bipartite graphs.

\begin{thm}\label{smallkbipthm}
Let $t\ge s\ge k$, where $k\in\{2,3\}$. Let $G$ be a $k$-connected bipartite graph whose partition classes have $s$ and $t$ vertices, and let $H\subset G$ be a minimum spanning $k$-connected subgraph. If $e(H)=kt$, then
\[
mc_k(G)=e(G)-kt+1.
\]
Moreover, if $s=t$, then the result holds for $k\in\{2,3,4,5\}$.
\end{thm}

Our proofs of Theorems \ref{smallkthm} and \ref{smallkbipthm} will consider multisets of colours at the vertices of an edge-coloured graph $G$. Let $G$ be given an edge-colouring $\phi:E(G)\to[r]$. For $v\in V(G)$, let $C(v)$ denote the multiset of the colours of the edges at $v$. For simplicity, we ignore commas and brackets for a multiset, and we use a superscript to denote the multiplicity of a colour, where a superscript of zero means that the colour is not present. For example, $C(v)=1^22^43^04=1^22^44$ means that the number of edges in colours $1,2,3,4$ at $v$ are $2,4,0$ and $1$. If we list some elements of a multiset, followed by $\ast$, we mean that the remaining elements of the multiset occur with multiplicity 1, and are different from those listed. For example, $C(v)=1^22^43^04\ast$ (resp.~$C(v)=1^22^44\ast$) means that there are a further $\deg_G(v)-7$ colours at $v$ which are different from colours $1,2,3,4$ (resp.~colours $1,2,4$). Containment and intersection relations between multisets are naturally defined. For $u,v\in V(G)$, we write $C(u)\subset C(v)$ if, for every colour in $C(u)$, its multiplicity in $C(u)$ is at most its multiplicity in $C(v)$. Also, $C(u)\cap C(v)$ is the multiset which contains all colours common to $C(u)$ and $C(v)$, where the multiplicity of a colour in $C(u)\cap C(v)$ is the minimum of its multiplicities in $C(u)$ and $C(v)$.

\begin{proof}[Proof of Theorem \ref{smallkthm}]
Since $e(H)=\lceil\frac{kn}{2}\rceil$, Observation \ref{LBobs} implies the lower bound $mc_k(G)\ge e(G)- \lceil  \frac{kn}{2}\rceil + 1$. We prove the upper bound $mc_k(G)\le e(G)- \lceil  \frac{kn}{2}\rceil + 1$, and we remark that the condition $e(H)=\lceil\frac{kn}{2}\rceil$ will not be required. Let $G$ be given a monochromatic $k$-connected colouring, with exactly $r$ colours. By Observation \ref{UBobs}, we may assume that every colour $i\in[r]$ induces exactly one non-trivial component. We prove that $r\le e(G)- \left\lceil  \frac{kn}{2}\right\rceil + 1$. Let $F$ be the spanning subgraph of $G$ obtained by deleting all edges in monochromatic components on two vertices, and note that these deleted edges have distinct colours. Suppose that $F$ uses exactly $r'\le r$ colours, say $1,\dots,r'$. For a set of colours $A\subset [r']$, let $\bar{A}=[r']\setminus A$. Let $G_A$ be the spanning subgraph of $F$ such that $E(G_A)$ is the set of all edges with a colour in $A$. For simplicity, we write $G_i$ for $G_{\{i\}}$. Then we can easily observe the following.
\begin{itemize}
\item[(O1)] Each monochromatic component of $F$ has at least two edges.
\item[(O2)] Any two vertices of $F$ are connected by $k-1$ disjoint monochromatic super-paths. Thus, $F$ is $(k-1)$-connected.
\item[(O3)] $\delta(F)\ge k$.
\item[(O4)] Let $vw\in E(F)$ be such that the colour of $vw$ appears exactly once at the vertex $v$. Then $\deg_{G_A}(w)\ge k+1$, where $A$ is the set of colours appearing at $v$.
\end{itemize}
Indeed, (O1) is from the definition of $F$, and (O2) is from the fact that the colouring of $G$ is monochromatic $k$-connected. To see (O3), let $x\in V(F)$ be any vertex. By (O2), there exists $y\in V(F)$ with $xy\in E(F)$, and $k-1$ disjoint super-paths from $x$ to $y$. Thus, $\deg_F(x)\ge k$. To see (O4), note that by (O1), there exists a vertex $x\in V(F)\setminus\{v,w\}$ such that $vw$ and $wx$ have the same colour, say $i\in A$. By (O2), $v$ and $w$ are connected by $k-1$ disjoint monochromatic super-paths, using colours from $A\setminus\{i\}$. Thus $\deg_{G_A}(w)\ge k+1$.

In the rest of the proof, we will show that
\begin{equation}\label{smallkeq1}
2e(F) \ge {kn}+2(r'-1),
\end{equation}
since \eqref{smallkeq1} implies $e(G) =  e(F)+(r-r')\ge \frac{kn}{2}+(r'-1)+(r-r')$, so that $r\le e(G)-\lceil \frac{kn}{2}\rceil +1$. Note that (\ref{smallkeq1}) holds for $r'=1$, since (O3) gives $2e(F)\ge kn$.  From now on, we let $r'\ge 2$.
\\[1ex]
\emph{Case 1.} Every vertex of $F$ is incident to every colour in $[r']$.\\[1ex]
\indent By the choice of the colouring and the case assumption, we have $e(G_i)\ge n-1$ for all $i\in [r']$.
Thus $2e(F)\ge 2r'(n-1)$. For $r'\ge \lceil\frac{k}{2}\rceil+1$,  (\ref{smallkeq1}) holds, since
\begin{align*}
2r'(n-1)- \left( kn+ 2(r'-1)\right) &= (2n-4)r'-kn+2 \ge (2n-4)\Big(\frac{k}{2}+1\Big)-kn+2 \\
&=2n-2k-2 \ge 0.
\end{align*}

Next, we consider the case when $k$ is odd and $r'=\frac{k+1}{2}$. If $\delta(F)\ge k+1$, then $2e(F)\ge (k+1)n\ge kn+k+1>kn+2(r'-1)$. Otherwise, there exists $u\in V(F)$ with $\deg_F(u)=k$. For $i\in [r']$, we estimate $e(G_i)$ as follows. If $\deg_{G_i}(u)\ge 2$, let $N_{G_i}(u)=\{v,v_1,\dots,v_t\}$, where $t=\deg_{G_i}(u)-1\ge 1$. By (O2), there exist $k-1$ disjoint monochromatic super-paths connecting $u$ and $v$, where exactly $t$ of these super-paths have colour $i$, and they use $v_1,\dots, v_t$. Deleting $uv_1,\dots,uv_t$ gives a tree $T'$ using colour $i$, and $uv\in E(T')$. We may then extend $T'$ to a spanning tree $T$ of $G_i$, so that $T'\subset T$. Note that $uv_1,\dots,uv_t\not\in E(T)\setminus E(T')$, since adding any of $uv_1,\dots,uv_t$ to $T'$ creates a cycle. Thus, $e(G_i)\ge e(T)+t=n-2+\deg_{G_i}(u)$. The inequality $e(G_i)\ge n-2+\deg_{G_i}(u)$ also holds if $\deg_{G_i}(u)=1$. Now,
\begin{align*}
2e(F) &=2\sum_{i=1}^{r'} e(G_i)\ge 2r'(n-2)+2\sum_{i=1}^{r'} \deg_{G_i}(u)=(k+1)(n-2)+2k\\
&=kn+n-2\ge kn+k-1=kn+2(r'-1),
\end{align*}
so again (\ref{smallkeq1}) holds.

The remaining cases are $(k,r')=(4,2),(5,2)$, so we are now required to prove that $2e(F)\ge kn+2$. Suppose to the contrary that $2e(F)\le kn+1$. From (O3), we have $\delta(F)\ge k$, and thus $2e(F)\ge kn$. This means that either $F$ is $k$-regular, or $k=5$ and $F$ has degree sequence $5,\dots,5,6$. Let $S$ be the set of vertices of $F$ with degree $k$, and in the latter case, let $w$ be the vertex with degree $6$. For $v\in V(F)$, let $C(v)$ denote the multiset of colours incident to $v$ in $F$. We have some further observations.

\begin{itemize}
\item[(O5)] $C(u)=C(v)$ for any two vertices $u,v\in S$.
\item[(O6)] $C(u)\subset C(w)$ for every $u\in S$.
\item[(O7)] No colour induces a Hamilton cycle in $F$.
\end{itemize}

Indeed, (O5) clearly holds by (O2) if $uv\in E(F)$. If $uv\not\in E(F)$, then by (O2), there are $k-1\ge 3$ disjoint super-paths connecting $u$ and $v$, so that all vertices of one super-path are in $S$. This implies that (O5) holds for $u$ and $v$. For (O6), if $v\in S$ and $v\in N_F(w)$, then for all $u\in S$, we have $C(u)=C(v)\subset C(w)$ by (O5) and (O2). For (O7), suppose that some colour induces a Hamilton cycle $C$ in $F$. Choose a vertex $u\in S$, and $v\in N_C(u)$. By (O2), there are $k-1\ge 3$ disjoint monochromatic super-paths connecting $u$ and $v$, one of which must be $C-uv$ which contains every vertex of $F$. This is a contradiction.

Now, suppose that $F$ is $k$-regular. If there exists $v\in V(F)$ incident with exactly one edge of colour $1$, then (O4) implies that there exists another vertex $x\in V(F)\setminus\{v\}$ with $\deg_F(x)\ge k+1$, a contradiction. Thus for all $u\in V(F)$, by (O5), we may assume that $C(u)=1^22^2$ for $k=4$, and $C(u)=1^22^3$ for $k=5$. In either case, colour $1$ induces a Hamilton cycle in $G$, which contradicts (O7).

Finally, suppose that $k=5$ and $F$ has degree sequence $5,\dots,5,6$. If there exists a vertex of $S=V(F)\setminus\{w\}$ incident with exactly one edge of colour $1$, then by (O5), the same holds for every other vertex of $S$. By (O4), it follows that every edge incident to $w$ has colour $1$, a contradiction to the case assumption. Thus by (O5), we may assume that $C(u)=1^22^3$ for all $u\in S$. Since $\sum_{x\in V(F)}\deg_{G_1}(x)$ is even, this means that the total number of $1$s over all $C(x)$, for $x\in V(F)$, is even. Thus by (O6), we have $C(w)=1^22^4$. But then, colour $1$ induces a Hamilton cycle, which again contradicts (O7).\\[1ex]
\noindent\emph{Case 2.} There is a vertex $v\in V(F)$ such that $\deg_F^c(v)=\delta^c(F)<r'$.\\[1ex]
\indent Let $\deg_F^c(v)=\delta^c(F)=p<r'$, and the set of colours of the edges incident to $v$ be exactly $A=[p]$.
Note that (O3) implies $\deg_{G_A}(v)\ge k$. By (O2), for every $x\in V(F)\setminus\{v\}$, there exist $k-1$ disjoint monochromatic super-paths from $v$ to $x$, using colours from $A$. Thus $\deg_{G_A}(x)\ge k-1$, and $\delta(G_A)\ge k-1$. Let $U=\{x\in V(F):\deg_{G_A}(x)=k-1\}$. For any $w\in V(F)$, let $Q_w=\{y\in N_F(w):$ The colour of $wy$ occurs uniquely at $w\}$, $q_w=|Q_w|$, and $R_w=N_F(w)\setminus Q_w$. Note that $0\le q_w\le \deg_F(w)$. Clearly, we have $\deg_F(v)\ge q_v+2(p-q_v)$.\\[1ex]
\noindent\emph{Subcase 2.1.} $U=\emptyset$.\\[1ex]
\indent We have $\delta(G_A)\ge k$. By (O1), we have $e(G_{\bar{A}})\ge 2(r'-p)$. Using (O4),
\begin{align}
2e(F) &= 2e(G_A)+2e(G_{\bar{A}})\ge \deg_F(v)+q_v(k+1)+(n-1-q_v)k+4(r'-p)\nonumber\\
&\ge q_v+2(p-q_v)+q_v(k+1)+(n-1-q_v)k+4(r'-p)\nonumber\\
&= kn+2(r'-1)+2(r'-p)+2-k.\label{sc2.1eq1}
\end{align}
If $p\le r'-2$, or $p=r'-1$ and $k\in\{2,3,4\}$, then (\ref{smallkeq1}) holds. Now suppose that $p=r'-1$ and $k=5$, so that $A=[r'-1]$ and $\bar{A}=\{r'\}$. Then (\ref{sc2.1eq1}) becomes
\begin{equation}
2e(F)\ge 5n+2(r'-1)-1. \label{sc2.1eq2}
\end{equation}
Suppose that equality holds in (\ref{sc2.1eq2}). Then $e(G_{\bar{A}})=2$, so the only non-trivial component of $G_{\bar{A}}$ is a path of length $2$ in colour $r'$, say $x_1x_2x_3$.
\begin{clm}\label{sc2.1clm}
For $w\in V(F)$, we have the following.
\begin{enumerate}
\item[(a)] If $w\not\in\{x_1,x_2,x_3\}$, then $\deg_{G_A}(y)=6$ for all $y\in Q_w$, and $\deg_{G_A}(z)=5$ for all $z\in V(F)\setminus(\{w\}\cup Q_w)$.
\item[(b)] $\deg_{G_A}(w)\in\{5,6\}$.
\item[(c)] If $w\not\in\{x_1,x_2,x_3\}$, then the colours appearing at $w$ are exactly $A=[r'-1]$. If $w\in\{x_1,x_2,x_3\}$, then at least $r'-2$ colours of $A$ appear at $w$. In either case, for any colour of $A$ that appears at $w$, the colour appears on either one or two edges incident to $w$.
\end{enumerate}
\end{clm}
\begin{proof}
(a) The assertion holds for $v$, since we have $2e(G_A)=\deg_F(v)+6q_v+5(n-1-q_v)$ by equality in (\ref{sc2.1eq1}). Now, let $w\in V(F)\setminus\{x_1,x_2,x_3\}$. Since $\delta^c(F)=p=r'-1$, and $w$ is not incident with colour $r'$, the set of colours at $w$ is exactly $A=[r'-1]$. Thus we may apply the same argument from the beginning of Case 2 with $w$ in place of $v$, so the assertion holds for $w$. \\[1ex]
\indent(b) Since $n\ge 6$ by (O3), this is immediate by using (a) on some vertex $u\in V(F)\setminus\{x_1,x_2,x_3,w\}$.\\[1ex]
\indent(c) The first assertion is mentioned in the proof of (a), and since $\delta^c(F)=p=r'-1$, the second assertion is also clear. We prove the final assertion. Applying the argument from the beginning of Case 2 with any $w\in V(F)\setminus\{x_1,x_2,x_3\}$ in place of $v$, we have $\deg_F(w)=q_w+2(p-q_w)=q_w+2(r'-1-q_w)$ by equality in (\ref{sc2.1eq1}), and the final assertion holds for $w$. Now consider $x_1$. Since $\deg_{G_A}(x_1)\ge 5$ by (b), let $y\in N_{G_A}(x_1)\setminus\{x_2,x_3\}$. If $x_1\in Q_y$, then by (a), we have $\deg_{G_A}(x_1)=6$, with an edge $x_1z$ having the same colour as $x_1y$ for some $z\in V(F)\setminus\{x_1,y\}$. If $x_1\in R_y$, then by (a), we have $\deg_{G_A}(x_1)=5$. In either case, there are four remaining edges at $x_1$ in $G_A$, excluding $x_1y,x_1z$ if the former, and excluding $x_1y$ if the latter. By (O2), these four edges must use the same colours as some four edges at $y$, excluding $x_1y$. It is then easy to see that the assertion for $y$ implies that the assertion holds for $x_1$. Similar arguments hold for $x_2,x_3$, and we are done.
\end{proof}

($\ast$) Now, we have $\deg_F(v)\in\{5,6\}$ by Claim \ref{sc2.1clm}(b). Suppose first that $q_v\ge 3$ (so that $r'\ge 4$). Let $v_1,v_2,v_3\in Q_v$, and $u\in N_F(v)\setminus\{v_1,v_2,v_3\}$. By (O2), there are four disjoint monochromatic super-paths connecting $v$ and $u$. We may assume that two such paths are $P,P'$, with colours $1$ and $2$, and using $v_1,v_2$ respectively. Using Claim \ref{sc2.1clm}(a) with $v$, and (O2) with $v,u$, it is easy to check that, whether $u\in Q_v$ or $u\in R_v$, there must exist exactly one edge in each of colours $1$ and $2$ at $u$. By Claim \ref{sc2.1clm}(c), and the fact that colours $1$ and $2$ each induce exactly one non-trivial component in $F$, it follows that $P$ and $P'$ are precisely the non-trivial components in colours $1$ and $2$. Moreover, we have $V(P), V(P')\supset V(F)\setminus\{x_1,x_2,x_3\}$, since every vertex of $V(F)\setminus\{x_1,x_2,x_3\}$ is incident to colours $1$ and $2$. Also, if $x_i\not\in V(P)$ for some $i$, then $x_i\in V(P')$, since $x_i$ must be incident to at least one of colours $1$ and $2$. Thus $V(P\cup P')=V(F)$, which is a contradiction.

Finally, let $q_v\le 2$. The final part of Claim \ref{sc2.1clm}(c) implies $|R_v|\ge 4$. Let $v_1,v_2,v_3,v_4\in R_v$, and $u\in N_F(v)\setminus\{v_1,v_2,v_3,v_4\}$. Assume that $vv_1,vv_2$ have colour $1$, and $vv_3,vv_4$ have colour $2$. Suppose that $\deg_F(v)=5$. Then by (O2), there are four disjoint monochromatic super-paths connecting $v$ and $u$, say $P,P',N,N'$, using $v_1,v_2,v_3,v_4$ respectively, so that $P,P'$ have colour $1$, and $N, N'$ have colour $2$. We have two cycles, $C=P\cup P'$ using colour $1$, and $D=N\cup N'$ using colour $2$. Using Claim \ref{sc2.1clm}(c) as before, we have $V(C),V(D)\supset V(F)\setminus\{x_1,x_2,x_3\}$. We may assume that $v_1\not\in\{x_1,x_2,x_3\}$. But then $v_1$ would be an internal vertex of $N$ or $N'$, a contradiction. At this point, we have always obtained a contradiction for the case $\deg_F(v)=5$. Now suppose that $\deg_F(v)=6$. We may assume that $v_1\not\in\{x_1,x_2,x_3\}$. By Claim \ref{sc2.1clm}(a) and (c), we have $\deg_F(v_1)=\deg_{G_A}(v_1)=5$. Thus we may apply the argument from the point ($\ast$) with $v_1$ in place of $v$ to obtain similar contradictions.

We conclude that equality cannot hold in (\ref{sc2.1eq2}). Therefore, (\ref{smallkeq1}) holds.\\[1ex]
\noindent\emph{Subcase 2.2.} $U\neq\emptyset$.\\[1ex]
\indent Recall that the vertex $v$ is incident to the set of colours $A=[p]$.
\begin{clm}\label{sc2.2clm}
$e(G_{\bar{A}})\ge |U|+k-2$.
\end{clm}

\begin{proof}
Consider the components of $G_{\bar{A}}$. Since $\delta(F)\ge k$ by (O3), we see that every vertex $x\in U$ must be incident to an edge with a colour in $\bar{A}$, so that $x$ belongs to exactly one non-trivial component of $G_{\bar{A}}$. Let $L_1,\dots,L_s$ be all the components of $G_{\bar{A}}$ that contain a vertex of $U$, for some $s\ge 1$. Note that we have the partition $U=(V(L_1)\cap U)\dcup\cdots\dcup(V(L_s)\cap U)$, and every $L_j$ has at least three vertices. Also, $v$ is an isolated vertex of $G_{\bar{A}}$, so $v$ does not belong to any of $L_1,\dots,L_s$, and $V(F)\setminus V(L_j)\neq\emptyset$ for $1\le j\le s$.

For $1\le j\le s$, we claim that if $x\in V(L_j)\cap U$, then $x$ is not incident to an edge of $E_F(V(L_j),V(F)\setminus V(L_j))$. Otherwise, suppose that $xy\in E_F(V(L_j),V(F)\setminus V(L_j))$ for some $y\in V(F)\setminus V(L_j)$. Since $L_j$ is a component of $G_{\bar{A}}$, the colour of $xy$ is in $A$. By (O2), there exist $k-1$ monochromatic super-paths connecting $x$ and $y$ in $F$. Since $\deg_{G_A}(x)=k-1$, one of the super-paths uses a colour in $\bar{A}$. This contradicts that $L_j$ is a component of $G_{\bar{A}}$, and the claim holds. Now, let $C\subset V(L_j)$ be the set of vertices that are incident to some edge of $E_F(V(L_j),V(F)\setminus V(L_j))$. By the definition of $L_j$, there exists $u\in V(L_j)\cap U$. By the claim, we see that $u\not\in C$. Also, note that $|C|\ge k-1$. Otherwise $C$ would be a cut-set of $F$ with at most $k-2$ vertices which separates $u$ from $V(F)\setminus V(L_j)$, and this contradicts that $F$ is $(k-1)$-connected by (O2). Finally, no vertex of $C$ can be in $U$. We conclude that $|V(L_j)\setminus U|\ge k-1$ for $1\le j\le s$.

Now, we have
\begin{align*}
e(G_{\bar{A}}) &\ge \sum_{j=1}^s e(L_j)\ge \sum_{j=1}^s (|V(L_j)|-1)=\sum_{j=1}^s (|V(L_j)\cap U|+|V(L_j)\setminus U|-1)\\
&\ge \sum_{j=1}^s (|V(L_j)\cap U|+k-2)= |U|+s(k-2)\ge |U|+k-2.
\end{align*}
This completes the proof of Claim \ref{sc2.2clm}.
\end{proof}

Now we complete the proof of Subcase 2.2. Note that
\begin{align*}
2e(G_A) &\ge \deg_F(v)+ q_v(k+1)+ (n-1-q_v-|U|)k + |U|(k-1)\\
&=\deg_F(v)+q_v+kn-k-|U|,
\end{align*}
and $e(G_{\bar{A}})\ge 2(r'-p)$ by (O1). Using Claim \ref{sc2.2clm}, a similar calculation as (\ref{sc2.1eq1}) gives
\begin{align*}
2e(F) &= 2e(G_A)+2e(G_{\bar{A}})\\
&\ge (\deg_F(v)+q_v+kn-k-|U|) +(|U|+k-2)+2(r'-p)\\
& = kn+2(r'-1)+(\deg_F(v)+q_v-2p)\\
& \ge kn+2(r'-1),
\end{align*}
so (\ref{smallkeq1}) holds.\\[1ex]
\indent The proof of Theorem \ref{smallkthm} is complete.
\end{proof}

\begin{proof}[Proof of Theorem \ref{smallkbipthm}]
We note that the ``moreover'' part follows immediately from Theorem \ref{smallkthm}. From now on, let $t\ge s\ge k$, where $k\in\{2,3\}$. Let $G$ be a bipartite graph as in the statement of the theorem. The lower bound $mc_k(G)\ge e(G)-kt+1$ follows from Observation \ref{LBobs}. We prove the upper bound $mc_k(G)\le e(G)-kt+1$.

We proceed by induction on $s+t$. The base case is $s=t=k$, and thus $G=K_{k,k}$. We may obtain $mc_k(K_{k,k})=1$ from Theorem \ref{gengraphsthm1}(b) or Theorem \ref{smallkthm}. Alternatively, suppose that $K_{k,k}$, with classes $A$ and $B$, is given a colouring, using at least two colours. We may assume that there exist $x\in A$ and $u,v\in B$ such that $ux$ and $vx$ have different colours. Then, to have $k$ disjoint paths connecting $u$ and $v$, one path must be $uxv$, which is not monochromatic. Thus, the colouring is not a monochromatic $k$-connected colouring of $K_{k,k}$, and $mc_k(K_{k,k})=1$.

Now, let $t\ge s\ge k$ and $s+t>2k$, so that $t>k$. Suppose that the upper bound holds for all $t'\ge s'\ge k$ with $s'+t'<s+t$. Let $G$ be a $k$-connected bipartite graph with classes $X$ and $Y$, where $|X|=s$ and $|Y|=t$. Suppose that $G$ is given a monochromatic $k$-connected colouring, using exactly $r$ colours. We prove that $r\le e(G)-kt+1$. Let the colours be $1,\dots, r$. For a non-empty subset $Y'\subset Y$, let $G[X,Y']$ denote the bipartite subgraph of $G$ induced by $X\cup Y'$. Let $b(Y')$ be the number of colours that are incident to at least one vertex of $Y'$. That is, we have $b(Y')=\big|\bigcup_{y\in Y'} N^c_G(y)\big|$. For $v\in Y$, we write $b(v)$ for $b(\{v\})$ for simplicity. Let $C(v)$ denote the multiset of colours incident to $v\in Y$. We have the following observations.
\begin{itemize}
\item[(O1)] For all $u,v\in Y$, there are $k$ edges incident to $u$ which have the same multiset of colours as $k$ edges incident to $v$. That is, $|C(u)\cap C(v)|\ge k$.
\item[(O2)] For all non-empty $Y'\varsubsetneq Y$ and $v\in Y\setminus Y'$, there are at most $\deg_G(v)-k$ colours incident to $v$ which are not used in $G[X,Y']$.
\item[(O3)] For all $v\in Y$, we have $b(v)\le \deg_G(v)-1$.
\end{itemize}

Indeed, (O1) is clear since there are $k$ disjoint monochromatic super-paths connecting $u$ and $v$. (O2) is also clear by taking $u\in Y'$ and using (O1). To see (O3), suppose that $b(v)=\deg_G(v)$ for some $v\in Y$. Then note that the colouring on $G'=G[X,Y\setminus\{v\}]$ is monochromatic $k$-connected, since any monochromatic super-path connecting two vertices of $G'$ cannot use $v$. By (O1) and induction, we have
\[
r\le
\left\{
\begin{array}{l@{\quad\:}l}
e(G')-k(t-1)+1+(\deg_G(v)-k)=e(G)-kt+1, & \textup{if $s<t$,}\\[0.5ex]
e(G')-ks+1+(\deg_G(v)-k)< e(G)-kt+1, & \textup{if $s=t$,}
\end{array}
\right.
\]
and we are done.

Now, let $u\in Y$. By (O3), we have $b(u)\le \deg_G(u)-1$. Taking $v\in Y\setminus\{u\}$ and applying (O2) gives $b(\{u,v\})\le \deg_G(u)-1+\deg_G(v)-k$. Repeating successively with the vertices of $Y\setminus\{u,v\}$ gives
\begin{equation}\label{smallkbipeq1}
r=b(Y)\le \sum_{y\in Y}\deg_G(y)-1-k(t-1)=e(G)-kt+k-1.
\end{equation}
We are done by induction for the case $k=2$. For the rest of the proof, let $k=3$. Assume on the contrary that equality in (\ref{smallkbipeq1}) holds. That is,  $r=e(G)-3t+2$. Since we may obtain (\ref{smallkbipeq1}) by starting with any vertex of $Y$ instead of $u$, it follows that $b(y)=\deg_G(y)-1$ for all $y\in Y$. That is, $C(y)$ has exactly one colour occurring twice, and $\deg_G(y)-2$ other colours each occurring once, for all $y\in Y$. Suppose that there exist $w,z\in Y$ such that $C(w)=a^2\ast$ and $C(z)=b^2\ast$, where $a\neq b$. Since $|C(w)\cap C(z)|\ge k$ by (O1), we have $C(w)=a^2b^pc_1\cdots c_{3-p-q}\ast$ and $C(z)=b^2a^qc_1\cdots c_{3-p-q}\ast$, for some $p,q\in\{0,1\}$ and $c_1,\dots,c_{3-p-q}\in[r]\setminus\{a,b\}$. Thus,
\begin{align*}
b(\{w,z\}) &\le (5-p-q)+\deg_G(w)-(5-q)+\deg_G(z)-(5-p)\\
&=\deg_G(w)+\deg_G(z)-5.
\end{align*}
Again, we may successively add vertices of $Y\setminus\{w,z\}$ to obtain $r=b(Y)\le \sum_{y\in Y}\deg_G(y)-5-3(t-2)=e(G)-3t+1$, a contradiction. Hence, we may assume that $C(y)=1^2\ast$ for all $y\in Y$. Now, choose $x\in X$ and $y\in Y$ such that $xy$ has colour $1$. There exist two disjoint monochromatic super-paths with odd length connecting $x$ and $y$, both of which must have colour $1$. But then, there are three edges at $y$ with colour $1$, a contradiction. The case $k=3$ follows by induction.
\end{proof}


\section{Complete graphs and complete bipartite graphs}\label{compsect}

In this final section, we shall use the results of Sections \ref{gensect} and \ref{smallksect} to deduce some results about the monochromatic $k$-connection number of complete graphs and complete bipartite graphs. We also propose some further conjectures.

For the case of the complete graph $K_n$, Conjecture \ref{mckconj1} becomes the following.

\begin{conj}\label{Knconj}
Let $n>k\ge 2$. We have
\[
mc_k(K_n)={n\choose 2}-\bigg\lceil\frac{kn}{2}\bigg\rceil+1.
\]
\end{conj}

From our previous results, we may deduce the following partial result of Conjecture \ref{Knconj}.

\begin{thm}\label{Knthm}
\indent
\begin{enumerate}
\item[(a)] Let $n>k\ge 2$. We have
\[
{n\choose 2}-\bigg\lceil\frac{kn}{2}\bigg\rceil+1\le mc_k(K_n) \le {n\choose 2}- \bigg\lceil  \frac{(k-1)n(n-1)}{2(n-2)}\bigg\rceil + 1.
\]
\item[(b)] Let $n > k$, where $k \in\{2, 3, 4, 5\}$. We have
\[
mc_k(K_n)={n\choose 2}-\bigg\lceil\frac{kn}{2}\bigg\rceil+1.
\]
\end{enumerate}
\end{thm}

\begin{proof}
(a) The lower bound follows from Observation \ref{LBobs} and Theorem \ref{Harthm}, and the upper bound follows from Theorem \ref{gengraphsthm1}(a).\\[1ex]
\indent (b) Immediate by Theorem \ref{smallkthm}.
\end{proof}

Now, we consider the complete bipartite graph $K_{s,t}$, where $t\ge s\ge k\ge 2$. Conjecture \ref{mckbipconj1} applies for $K_{s,t}$. For the case of $K_{t,t}$, if $H$ is a minimum spanning $k$-connected subgraph of $K_{t,t}$, then $e(H)=kt=\big\lceil\frac{k(t+t)}{2}\big\rceil$, and Conjecture \ref{mckconj1} applies. Thus, we have the following conjecture for complete bipartite graphs.

\begin{conj}\label{Kstconj}
\indent
\begin{enumerate}
\item[(a)] Let $t\ge s\ge k\ge 2$. We have
\[
mc_k(K_{s,t})=st-kt+h_k(K_{s,t}).
\]
\item[(b)] Let $t\ge k\ge 2$. We have
\[
mc_k(K_{t,t})=t^2-kt+1.
\]
\end{enumerate}
\end{conj}

Using our previous results, we have the following partial result of Conjecture \ref{Kstconj}.

\begin{thm}\label{Kstthm}
\indent
\begin{enumerate}
\item[(a)] Let $t\ge s\ge k\ge 2$. We have
\begin{equation}
st-kt+1\le mc_k(K_{s,t}) \le st- \bigg\lceil  \frac{kt}{2}\bigg\rceil + 1.\label{compbipthmeq1}
\end{equation}
Moreover, for $s=k\ge 2$ or $k\in\{2,3\}$, we have
\begin{equation}
mc_k(K_{s,t})=st-kt+1.\label{compbipthmeq2}
\end{equation}
\item[(b)] Let $t\ge k\ge 2$. We have
\begin{equation}
t^2-kt+1\le mc_k(K_{t,t}) \le t^2- \bigg\lceil \bigg(k-\frac{1}{2}\bigg)t\bigg\rceil + 1.\label{compbipthmeq3}
\end{equation}
Moreover, for $k\in\{2,3,4,5\}$, we have
\begin{equation}
mc_k(K_{t,t})=t^2-kt+1.\label{compbipthmeq4}
\end{equation}
\end{enumerate}
\end{thm}

\begin{proof}
The results (\ref{compbipthmeq2}) where $k\in\{2,3\}$, and (\ref{compbipthmeq4}), follow immediately from Theorem \ref{smallkbipthm}. We prove (\ref{compbipthmeq2}) when $s=k\ge 2$. That is, $mc_k(K_{k,t})=1$. Let $X$ and $Y$ be the classes of $K_{k,t}$, where $|X|=k$ and $|Y|=t$. Suppose that there exists a monochromatic $k$-connected colouring of $K_{k,t}$, using at least two colours. If there exist $w\in X$ and $u,v\in Y$ such that $wu$ and $wv$ have different colours, then we do not have $k$ disjoint monochromatic paths connecting $u$ and $v$, since one path must be $uwv$. Thus, every vertex of $X$ must be incident with one colour. Now, we may assume that there exist $x,x'\in X$ and $y\in Y$ such that $xy$ and $x'y$ have colours $1$ and $2$. To have $k$ disjoint monochromatic paths connecting $x$ and $y$, one path must be $xy'x'y$ for some $y'\in Y\setminus\{y\}$, which has colour $2$. But then, $x$ is incident with colours $1$ and $2$, a contradiction.

Now, for the inequalities (\ref{compbipthmeq1}) and (\ref{compbipthmeq3}), both lower bounds follow from Observation \ref{LBobs}, and the fact that $e(H)=kt$ if $H$ is a minimum spanning $k$-connected subgraph of $K_{s,t}$ or $K_{t,t}$. We prove the two upper bounds. Applying Theorem \ref{gengraphsthm1}(a) for $G=K_{s,t}$, we have
\[
mc_k(K_{s,t})\le st-\bigg\lceil\frac{k{s+t\choose 2}-st}{s+t-2}\bigg\rceil+1=st-\bigg\lceil\frac{k(s+t)(s+t-1)-2st}{2(s+t-2)}\bigg\rceil+1.
\]
For the upper bound of (\ref{compbipthmeq1}), it suffices to show that
\[
\frac{k(s+t)(s+t-1)-2st}{2(s+t-2)}\ge\frac{kt}{2}.
\]
This inequality is equivalent to $k(st+s^2+t-s)\ge 2st$, which is true since $t\ge s\ge k\ge 2$.

\indent For the upper bound of (\ref{compbipthmeq3}), setting $s=t$, it suffices to show that
\[
\frac{2kt(2t-1)-2t^2}{2(2t-2)}\ge\bigg(k-\frac{1}{2}\bigg)t.
\]
This inequality is equivalent to $k\ge 1$, which is true.
\end{proof}

\section*{Acknowledgements}
Qingqiong Cai is supported by National Key Research and Development Program of China (No.~2022YFA1006400), and Fundamental Research Funds for the Central Universities (050-63231193). Shinya Fujita is supported by JSPS KAKENHI (No.~23K03202). Henry Liu is partially supported by National Natural Science Foundation of China (No.~11931002), and National Key Research and Development Program of China (No.~2020YFA0712500). Boram Park is supported under the framework of an international cooperation program managed by the National Research Foundation of Korea (NRF-2023K2A9A2A06059347).

Shinya Fujita and Boram Park acknowledge the generous hospitality of Sun Yat-sen University, Guangzhou, China. They were able to carry out part of this research with Henry Liu during their visits there.

\end{document}